\documentclass[12pt]{amsart}
\usepackage{amsmath,amsthm, amssymb}
\usepackage{graphicx,epic}
\usepackage{mathrsfs}
\usepackage{hyperref}
\usepackage{float}
\newtheorem{theorem}{Theorem}
\newtheorem{proposition}{Proposition}[section]
\newtheorem{lemma}[proposition]{Lemma}
\newtheorem{corollary}[proposition]{Corollary}

\newtheorem*{claim*}{Claim}
\newtheorem{definition}[proposition]{Definition}

\def\ie{{\em i.e.,\ }}
\def\eg{{\em e.g.\ }}
\def\SS{{\mathbb S}^1} 
\def\eps{\varepsilon}
\def\phi{\varphi}
\def\N{{\mathbb N}}

\def\R{{\mathbb R}}

\def\K{{\mathbb K}}
\newcommand {\CA}{{\mathcal A}}

\newcommand {\CL}{{\mathcal L}}

\newcommand {\CP}{{\mathcal P}}

\newcommand {\CR}{{\mathcal R}}

\newcommand {\CX}{{\mathcal X}}

\def\s{\sigma}
\def\l{\lambda}

\def\1{ {\hbox{{\it 1}} \!\! I} }

\def\?{{\bf ???}}

\def\al{\alpha}
\def\be{\beta}

\def\ga{\gamma}

\def\8{\infty}

\newcommand{\ol}{\overline}
\def\disp{}

\newcommand{\arc}[1]{\stackrel{\frown}{#1}}
\renewcommand{\S}{\Sigma}
\usepackage{float,color}
\newcommand{\wt}{\widetilde}

\newcommand{\BBone}{{1\!\!1}}

\newcommand{\w}{\mathbf{ w}}
\newcommand{\e}{\mathbf{ e}}

\theoremstyle{definition}
\newtheorem{remark}{Remark}

\begin{document}
\synctex =1
\title[Renormalization and thermodynamics]
{Renormalization, freezing phase transitions and Fibonacci
quasicrystals.}
\author{Henk Bruin and Renaud Leplaideur}
\date{Version of \today}
\thanks{}


\maketitle

\begin{abstract}
We examine the renormalization operator determined by the Fibonacci substitution. We exhibit a fixed point and determine its stable leaf (under iteration of the operator). Then, we study the 
thermodynamic formalism for potentials in this stable leaf, and prove they have a freezing phase transition, with ground state supported on the attracting quasi-crystal associated to the Fibonacci substitution. 
\end{abstract}

%

\section{Introduction}\label{sec:intro}
\subsection{Background}
The present paper studies phase transitions from an ergodic theory and dynamical system viewpoint. It investigates the relations between renormalization, substitutions and phase transition initiated in \cite{baraviera-leplaideur-lopes} and continued in \cite{BL}.

Phase transitions are an important topic in statistical mechanics and also in probability theory (see \eg \cite{georgii,grimmett}). 
The viewpoint presented here is different for several reasons. One of them is that, here, the geometry of the lattice is not relevant\footnote{and we only consider a one-dimensional lattice.}, whereas in statistical mechanics, 
the geometry of the lattice is the most relevant part. 

A {\em phase transition} is characterized by a lack of analyticity of the 
pressure function. 
This definition of phase transition is inspired by statistical mechanics 
and is now standard for dynamical systems, see \cite{bowen, ruelle, sinai}. 
Given a dynamical systems, say $(X,T)$, and a potential $\phi:X\to\R$, the pressure function is given by 
$$
\CP(\be):=\sup\left\{h_{\mu}(T)+\be \int\phi\,d\mu\right\},
$$
where the supremum is taken over the invariant probability measures $\mu$, $h_{\mu}(T)$ is the Kolmogorov entropy and $\be$ is a real parameter. 

For a uniformly hyperbolic dynamical system $(X,T)$ and 
a H\"older continuous potential $\phi$, the pressure function $\be\mapsto \CP(\be)$ is analytic (see \eg \cite{bowen, ruelle, keller}). Even if analyticity is usually considered as a very rigid property and thus quite rare, it turns out that proving non-analyticity for the pressure function is not so easy. 
Currently, this has become an important challenge in smooth ergodic theory to produce and study phase transitions, see \eg \cite{Makarov-Smirnov, coronel-Rivera-Letelier, BL, iommi-todd}. 

To observe phase transitions, one has to weaken hyperbolicity of the system
or of regularity of the potential; it is the latter one that we continue to investigate here. Our dynamical system is the full shift, which is uniformly hyperbolic. 
The first main question we want to investigate is thus what potentials $\phi$ will produce phase transitions. More precisely, we are looking for a machinery to produce potentials with phase transitions. 

The main purpose of \cite{baraviera-leplaideur-lopes} was to investigate possible relation between {\em renormalization} and phase transition. In the shift
space  $(\{0,1\}^{\N}, \s)$, a renormalization is a function $H$ satisfying an equality of the form 
\begin{equation}\label{equ1-renorm}
\s^{k}\circ H=H\circ \s.
\end{equation}
 The link with potentials was made in \cite{baraviera-leplaideur-lopes} by introducing a renormalization operator $\CR$ acting on potentials and related to a solution $H$ for \eqref{equ1-renorm}.

It is easy to check that constant length $k$ substitutions are 
solutions to the renormalization equation.
In \cite{BL}, we studied  the Thue-Morse case substitution,
which has constant length $2$. 
Here we investigate the Fibonacci substitution, which is not of constant length. Several reasons led us to study the Fibonacci case:

$\bullet$ Together with the Thue-Morse substitution, the Fibonacci substitutions is the most ``famous'' substitution and it has been well-studied. In particular, the dynamical properties of their respective attracting sets are well-known and this will be
used extensively in this paper for the Fibonacci shift. As a result, we were 
able to describe the relevant fixed point of renormalization exactly. 
Information of the left and right-special words in these attractors is 
a key ingredient to prove existence of a phase transition; it 
is a crucial issue in the relations between 
substitutions and phase transitions.  

$\bullet$ The type of phase transition we establish is a {\em freezing phase transition}. This means that beyond the phase transition (\ie for large $\be$), 
the pressure function is affine and equal to its asymptote, and the equilibrium state (\ie ground state) is the unique shift-invariant measure
supported on an aperiodic subshift space, sometimes called {\em quasi-crystal}.
One open question in statistical mechanics (see \cite{vanEnter?}) is whether freezing phase transitions can happen and whether {\em quasi-crystal ground state}
can be reach at {\em positive temperature}. An affirmative answer was given for the Thue-Morse quasi-crystal in \cite{BL}; we show here this also holds for the Fibonacci quasi-crystal. 

$\bullet$ 
We think that Fibonacci shift opens the door to study more cases. One natural question is whether any quasi-crystal can be reached as a ground state at positive temperature. In this context we emphasize that the Fibonacci substitution also has a Sturmian shift, that is, it is related to the irrational rotation with angle the golden mean $\ga:=\frac{1+\sqrt5}2$. 
We expect that the machinery developed here for the Fibonacci substitution can
be extended to the Sturmian shift associated to general irrational rotation
numbers 
(although those with bounded entries in the continued fraction expansion
will be the easiest), possibly to rotations on higher-dimensional tori, and also to more general substitutions.

\subsection{Results}

Let $\S = \{0,1\}^{\N}$ be the full shift space; points in $\S$ 
are sequences $x:=(x_{n})_{n\ge 0}$ or equivalently infinite {\em words} $x_{0}x_{1}\ldots$. 
Throughout we let $\ol x_j = 1-x_j$ denote the opposite symbol.
The dynamics is the left-shift 
$$
\s: x=x_{0}x_{1}x_{2}\ldots\mapsto x_{1}x_{2}\ldots.
$$
Given a word $w=w_{0}\ldots w_{n-1}$ of {\em length} $|w|=n$, the corresponding
{\em cylinder} (or {\em $n$-cylinder}) is the set of infinite words starting as $w_{0}\ldots w_{n-1}$. We use the notation
$C_n(x)=[x_0\dots x_{n-1}]$
for the $n$-cylinder containing $x=x_{0}x_{1}\ldots$ 
If $w=w_{0}\ldots w_{n-1}$ is a word with length $n$ and $w'=w'_{0}\ldots$ a word of any length, the {\em concatenation} $ww'$ is the word $w_{0}\ldots w_{n-1}w'_{0}\ldots$. 

The Fibonacci substitution on $\S$ is defined by: 
$$
H: \begin{cases}
0 \to 01\\
1 \to 0.
\end{cases}
$$
and extended for words by the concatenation rule $H(ww')=H(w)H(w')$.
It is convenient for us to count the Fibonacci numbers starting with index $-2$:
\begin{equation}\label{eq:Fibo}
F_{-2} = 1,\ F_{-1} = 0,\ F_0=1, \ F_1 = 1,\ F_2 = 2,\ F_{n+2} = F_{n+1} + F_n,
\end{equation}
We have
\begin{equation}\label{eq:Fiboa}
F_{n}^{a}:=|H^{n}(a)| = \begin{cases}
F_{n+1} & \text{ if } a = 0,\\
F_{n} & \text{ if } a = 1.
\end{cases}
\end{equation}
The Fibonacci substitution has a unique fixed point
$$
\rho = 0\ 1\ 0\ 01\ 010\ 01001\ 01001010\ 0100101001001\dots
$$
We define the orbit closure $\K = \overline{\cup_n \s^n(\rho)}$; it forms a subshift
of $(\S, \sigma)$ associated to $\rho$. More properties on $\K$ are given in Section~\ref{sec-H-K-R}.

\bigskip
We define the renormalization operator 
acting on potentials $V:\S\to \R$ by 
$$
(\CR V)(x)= \begin{cases}
 V\circ \s\circ H(x)+V\circ H(x) & \text{ if }x\in[0],\\
 V\circ H(x) & \text{ if }x\in[1].
\end{cases}
$$
We are interested in finding fixed points for $\CR$ and, where possible, studying their stable leaves,
\ie potentials converging to the fixed point under iterations of $\CR$. 
Contrary to the Thue-Morse substitution, 
the Fibonacci substitution is not of constant length. This is the source of several complications, in particular for the correct expression for $\CR^{n}$. 

For  $\al>0$, let $\CX_{\al}$ be the set of functions $V: \S \to \R$ 
such that $\disp V(x) \sim n^{-\al}$ if $d(x,\K)=2^{-n}$. 
More precisely, $\CX_{\al}$ is the set of functions $V$ such that:
\begin{enumerate}
\item $V$ is continuous and non-negative. 
\item There exist two continuous functions $g,h:\S \to \R$, satisfying $\disp h_{|\K}\equiv 0$ and $g>0$, such that 
$$
V(x) = \frac{g(x)}{n^\al} + \frac{h(x)}{n^\al} \quad \text{ when } \quad 
d(x,\K) = 2^{-n}.
$$
\end{enumerate}
We call $g$ the \textit{$\al$-density}, or just the \textit{density} 
of $V\in \CX_{\al}$. Continuity and the assumption $h_{|\K}\equiv 0$ imply 
that $h(x)/n^{\al} = o(n^{-\al})$.

Our first theorem achieves the existence of a fixed point for $\CR$ and shows that the germ of $V$ close to $\K$, \ie its $\al$-density, allows us to 
determine the stable leaf of that fixed point. 

Given a finite word $w$, let
$\kappa_a(w)$ denote the number of symbols $a\in\{0,1\}$ in $w$.
If $x \in \S \setminus \K$, we denote by $\wt\kappa_{a}(x)$ the number of symbols $a$ in the finite word $x_{0}\ldots x_{n-1}$ where $d(x,\K)=2^{-n}$. 

\begin{theorem} \label{theo-fixedpoint}
If $V \in \CX_\al$, with $\al$-density function $g$, then
$$
\lim_{k\to\8} \CR^{k}V(x) =
\begin{cases}
\quad \8 & \text{ for all  } x\in \S\setminus \K \text{ if } \al < 1; \\[2mm]
\quad 0 & \text{ for all  } x\in \S \text{ if } \al > 1; \\[2mm]
\quad \disp\int g \ d\mu_\K \cdot \wt V(x)
& \text{ for all  } x\in \S \text{ if } \al = 1,
\end{cases}
$$
where  $\wt V\in\CX_{1}$ is a fixed point for $\CR$, given by 
\begin{equation}\label{eq:tildeV}
\wt V(x) = 
\begin{cases}
 \log\left(\disp\frac{\wt\kappa_{0}(x)+\frac1\ga\wt\kappa_{1}(x)+\ga}{\disp\wt\kappa_{0}(x)+\frac1\ga\wt\kappa_{1}(x)+\ga-1}\right) & \text{ if } x \in [0];\\[4mm]
 \log\left(\disp\frac{\ga\wt\kappa_{0}(x)+\wt\kappa_{1}(x)+\ga^{2}}{\disp\ga\wt\kappa_{0}(x)+\wt\kappa_{1}(x)+\ga^{2}-1}\right) & \text{ if } x \in [1].
\end{cases}
\end{equation}
\end{theorem}
This precise expression of $\wt V$ corresponds to a $\alpha$-density
$\tilde g(x) = \ga^2/(2\ga-1)$ if $x \in [0] \cap K$ and $\tilde g(x) =
\ga/(2\ga-1)$ if $x \in [1] \cap \K$,
and $\int \wt V(x) d\mu_\K = 1$.

Our second theorem suggests that renormalization for potentials is a machinery to produce potentials with phase transition. 
We recall that a {\em  freezing phase transition} is characterized
by the fact that the pressure is of the form 
$$
\CP(\be)=a\be+b \quad \text{ for } \be\ge\be_{c}
$$
and that the equilibrium state is fixed for $\be\ge\be_c$.
The word ``freezing'' comes from the fact that in statistical mechanics 
$\be$ is the inverse of the temperature (so the temperature 
goes to $0$ as $\be\to+\8$) and that a {\em ground-state} 
is reached at positive temperature $1/\be_c$, see \cite{CLT, Dyson}. 

\begin{theorem}\label{theo-pt}
Any potential $\varphi:=-V$ with $V\in \CX_{1}$ admits a freezing phase transition at finite $\be$: there exists $\be_{c}>0$ such that 
\begin{enumerate}
\item for $0\le \be<\be_{c}$ the map $\CP(\be)$ is analytic, there exists a unique equilibrium state for $\be \varphi$ and this measure has full support;
\item for $\be>\be_{c}$, $\CP(\be)=0$ and  $\mu_{\K}$ is the unique 
equilibrium state for $\be \varphi$.  
\end{enumerate}
\end{theorem}

These two theorems explain a link between substitution, renormalization and phase transition on quasi-crystals: a substitution generates a quasi-crystal but also allows to define a renormalization operator acting on the potentials. This operator has some fixed point, and the stable leaf of that fixed point furnishes a family of potentials with freezing phase transition. 

\subsection{Outline of the paper}
In Section~\ref{sec-H-K-R} we recall and/or prove various
properties of the Fibonacci subshift and its special words.
We establish the form of $H^n$ and $\CR^nV$ for arbitrary $n$ and relate this
to (special words of) the Fibonacci shift.
In Section~\ref{sec-prooftheofix}, after  clarifying the role of accidents on
the computation of $\CR^nV$, we prove Theorem~\ref{theo-fixedpoint}.
Section~\ref{sec-proffthpt} deals with the thermodynamic formalism.
Following the strategy of \cite{leplaideur-synth} we specify and estimate the required (quite involved) quantities that are the core of the proof of Theorem~\ref{theo-pt}.

%
\section{Some more properties of $H$, $\K$ and $\CR$}\label{sec-H-K-R}
\subsection{The set \boldmath $\K$ as Sturmian subshift  \unboldmath}
In addition to being a substitution subshift, $(\K,\s)$ is the Sturmian subshift associated to the golden mean rotation, $T_{\ga}:x \mapsto x+\ga \pmod 1$. The golden mean is
$\ga=\frac{1+\sqrt5}2$  and it satisfies $\ga^{2}=\ga+1$. 

Fixing an orientation on the circle, 
let $\arc{ab}$ denote the arc of points between $a$ and $b$ in the circle 
in that orientation. 
If we consider the itinerary of $2\ga$ under the action of $T_{\ga}$ with the code $0$ if the point belongs to $\arc{0\ga}$ and $1$ if the point belongs to $\arc{\ga0}$ (see Figure~\ref{fig-codi-fibo}), we get $\rho$, the fixed point of the substitution.

\begin{figure}[htbp]
\begin{center}
\includegraphics[scale=0.5]{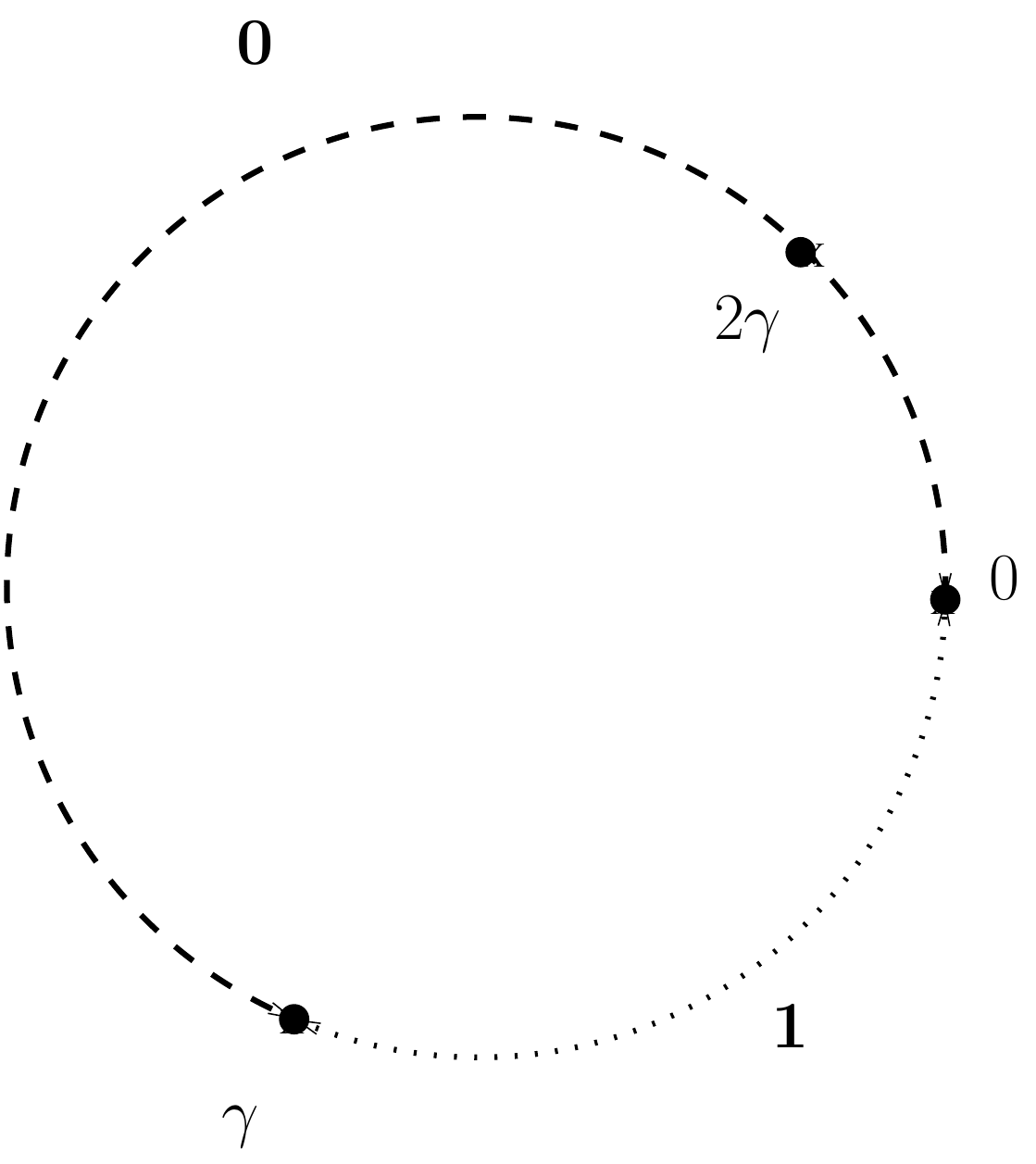}
\caption{Coding for Fibonacci Sturmian subshift.}
\label{fig-codi-fibo}
\end{center}
\end{figure}

There is an almost (\ie up to a countable set) one-to-one correspondence between points in $\K$ and codes of orbits of $(\SS, T_\ga)$,
 expressed by the commutative diagram
$$
\begin{array}{ccc}\SS & \stackrel{\disp T_\ga}{\longrightarrow} & \SS \\\pi\downarrow & \circlearrowleft & \downarrow\pi \\\K & \stackrel{\disp\s}{\longrightarrow} & \K\end{array}
$$
and $\pi$ is a bijection, except at points $T_\ga^{-n}(\ga) \in \SS$, $n \ge 0$.
Since Lebesgue measure is the unique $T_\ga$-invariant probability measure,
$\mu_\K := \text{Leb} \circ \pi^{-1}$ is the unique
invariant probability measure of $(\K,\s)$.  

We will use the same terminology for both $\K$ and $\SS$. 
For instance, a cylinder $C_{n}(x)$ for $x \in \SS$ is an interval\footnote{Some work has to be done to check that it actually is an interval.}, 
with the convention that $C_{n}(x)=\pi^{-1}(C_{n}(\pi(x)))$, and we may confuse 
a point $x \in \SS$ and its image $\pi(x) \in \K$. 

\begin{definition}\label{def-words}
Let $\CA_{\K}$ denote the set of finite words that appear in $\rho$.
 A word $\omega:=\omega_{0}\ldots \omega_{n-1}\in \CA_{\K}$ is said to be {\em left-special} if $0w$ and $1w$ both appear in $\CA_{\K}$. It is {\em right-special} if $w0$ and $w1$ both appear in $\CA_{\K}$.
A left and right-special word is called {\em bi-special}. A {\em special} word is either left-special or right-special. 
\end{definition}
 
Since $\rho$ has $n+1$ words of length $n$ 
(a characterization of Sturmian words), there is exactly one left-special 
and one right-special word of length $n$.
They are of the form $\rho_{0}\ldots \rho_{n-1}$ and $\rho_{n-1}\ldots \rho_{0}$
respectively, which can be seen from the forward itinerary
of $x \approx \ga$ and backward itinerary of $x \approx 0$ in the circle. 
Sometimes the left and right-special word merge into a single bi-special 
word $\omega$, but only
one of the two words $0\omega0$, and $1\omega1$ appears in $\CA_{\K}$, see \cite[Section 1]{arnoux-rauzy}, the construction of $\Gamma_{n+1}$ from $\Gamma_{n}$.

\begin{proposition}\label{prop-bispecialfibo}
Bi-special words in $\CA_{\K}$ are of the form $\rho_{0}\ldots \rho_{F_m-3}$
and for each $m \ge 3$, $\rho_{0}\ldots \rho_{F_m-3}$ is bi-special.
\end{proposition}

We prove this proposition at the end of Section~\ref{subsec:Hn}

\subsection{Results for \boldmath $H^{n}$ \unboldmath}\label{subsec:Hn}
We recall that $\kappa_{a}(w)$ is the number of symbol $a$ in the finite word $w$. 

\begin{lemma}\label{lem-lengthHn}
For any finite word $w$, the following recursive relations hold:
\begin{eqnarray*}
\kappa_0(H^n(w)) &=& F_n \kappa_0(w) + F_{n-1} \kappa_1(w);\\
\kappa_1(H^n(w)) &=& F_{n-1} \kappa_0(w) + F_{n-2} \kappa_1(w);\\
|H^n(w)| &=& F_{n+1} \kappa_0(w) + F_n \kappa_1(w) = |H^{n-1}(w)|+|H^{n-2}(w)|, 
\end{eqnarray*}
where $|H^{0}(w)|=|w|,\ |H^{1}(w)|=|H(w)|$.
\end{lemma}

Since we have defined $F_{-2} = 1$ and $F_{-1} = 0$, see \eqref{eq:Fibo},
these formulas hold for $n = 0$ and $n=1$ as well.

\begin{proof}
Since $H^n(0)$ contains $F_{n+1}$ zeroes and $F_{n-1}$ ones,
while  $H^n(0)$ contains $F_{n-1}$ zeroes and $F_{n-2}$ ones,
the first two lines follow from concatenation.
The third line is the sum of the first two, and naturally
the recursive relation follows from the same recursive relation
for Fibonacci numbers.
\end{proof}

Since $(\K, \s, \mu_\K)$ is uniquely ergodic, and isomorphic to 
$(\SS, T_\ga, \text{Leb})$, we immediately get that
\begin{equation}\label{eq:symfreq}
\lim_{n \to +\8} \frac{\kappa_a(H^n(w))}{ |H^n(w)| } =
\begin{cases} 
|\arc{0\ga}| = \frac1\ga & \text{ if } a = 0, \\
|\arc{\ga0}| = 1-\frac1\ga & \text{ if } a = 1. 
\end{cases}
\end{equation}

\begin{lemma}\label{lem-prop-stopbloc}
Assume that $x$ and $y$ have a maximal common prefix $w$.
Then $H^{n}(x)$ and $H^{n}(y)$ coincide for $T_{n}(w)+F_{n+2}-2$ digits, where $T_{n}(w)$ is defined by 
\begin{equation}\label{eq:Tnw}
T_{0}= |w|,\ T_{1}=|H(w)|,\ T_{n+2}(w) = T_{n+1}(w)+T_{n}(w).
\end{equation}
\end{lemma}

\begin{proof}
For $x = w0\dots$ and $y = w1\dots$, we find
\begin{align*}
\begin{array}{c} w \\ w \\\end{array}
& \begin{array}{|c|}\hline 0  \\1  \\\hline \end{array}
\ \stackrel{H}{\longrightarrow}
\begin{array}{cc} H(w) & 0\ \\ H(w) & 0\ \\\end{array}
\begin{array}{|c|}\hline 1 \\0 \\\hline \end{array}
\ \stackrel{H}{\longrightarrow}\begin{array}{cccc} H^2(w) & 0 & 1 & 0\ \\ H^2(w) & 0 & 1 & 0\ \end{array}
\begin{array}{|c|}\hline 0 \\ 1 \\\hline \end{array} \\[3mm]
& \stackrel{H}{\longrightarrow}
\begin{array}{ccccccc} H^3(w) & 0 & 1 & 0 & 0 & 1 & 0\ \\
H^3(w) & 0 & 1 & 0 & 0 & 1 & 0\ \end{array}
\begin{array}{|c|}\hline 1 \\0  \\\hline \end{array}
\ \stackrel{H}{\longrightarrow} \ \cdots
\end{align*}
where we used that $H(a)$ starts with $0$ for both $a=0$ and $a=1$.
We set $T_n(w) = |H^n(w)|$, then the recursive formula \eqref{eq:Tnw}
follows as in Lemma~\ref{lem-lengthHn}.

Iterating $H$ on the words $01$ and $10$, we find:
\begin{equation}\label{eq:Hn0110}
\begin{array}{|cc|}\hline 0 & 1 \\1 & 0 \\\hline \end{array}\stackrel{H}{\longrightarrow}\begin{array}{c} 0 \\0 \\\end{array}\begin{array}{|cc|}\hline 1 & 0 \\0 & 1 \\\hline \end{array}\stackrel{H}{\longrightarrow}\begin{array}{ccc}0 & 1 & 0 \\0 & 1 & 0\end{array}\begin{array}{|cc|}\hline 0 & 1 \\1 & 0 \\\hline \end{array}
\stackrel{H}{\longrightarrow}\begin{array}{cccccc}0 & 1 & 0 & 0 & 1 & 0 \\0 & 1 & 0 & 0 & 1 & 0\end{array}
\begin{array}{|cc|}\hline 1 & 0 \\0 & 1 \\\hline \end{array}\ .
\end{equation}
Thus $|H^n(10)| = |H^n(01| = F_{n+2}$ and the common prefix
of $H^n(10)$ and $H^n(01)$ has length $F_{n+2}-2$ is precisely the same as
the common block of $H^n(w0)$ and $H^n(w0)$ between $H^n(w)$ and
the first difference.

Therefore, $H^n(x)$ and $H^n(y)$
coincide for $T_n(w) + F_{n+2}-2$ digits.
\end{proof}

\begin{corollary}\label{cor-coinc-hn-rho}
For $x \in \K$ and $n\in\N$, $H^{n}(x)$ and $\rho$ coincide for at least $F_{n+3}-2$ digits if $x\in [0]$ and for at least $F_{n+2}-2$ digits if $x\in[1]$. 
\end{corollary}

\begin{proof}
If $x \in [0]$, then, by Lemma~\ref{lem-prop-stopbloc}, 
$H^n(x)$ coincides with $H^n(\rho) = \rho$
for at least $T_n(0) + F_{n+2}-2$ digits.
But $T_n(0) = |H^n(0)| = F_{n+1}$, so $T_n(0) + F_{n+2}-2 = F_{n+3}-2$.

If $x \in [1]$, then $H(x) \in [0]$ and the previous argument gives that 
$H^n(x)$ coincides with $H^n(\rho) = \rho$
for at least $F_{n+2}-2$ digits.
\end{proof}

\begin{proof}[Proof of Proposition~\ref{prop-bispecialfibo}]
We iterate the blocks $0\cdot01$, $0\cdot10$ and $1\cdot01$ under $H$:\\
$$
\begin{array}{|cccc|}\hline 0 & \cdot & 0 & 1 \\ 0 & \cdot & 1 & 0 \\ 1 & \cdot & 0 & 1 \\ \hline \end{array}
\stackrel{H}{\longrightarrow}
\begin{array}{|cc|}\hline 0 & 1 \\ 0 & 1 \\ & 0  \\ \hline \end{array}\
\begin{array}{c} 0 \\0 \\ 0 \end{array}\
\begin{array}{|cc|}\hline 1 & 0 \\ 0 & 1 \\  1 & 0 \\ \hline 
\end{array}\stackrel{H}{\longrightarrow}
\begin{array}{|ccc|}\hline \dots\!\!\! & 1 & 0 \\ \dots\!\!\! & 1 & 0 \\ & 0 & 1 \\ \hline \end{array}\
\begin{array}{ccc} 0 & 1 & 0 \\ 0 & 1 & 0 \\ 0 & 1 & 0 \end{array}\
\begin{array}{|cc|}
\hline 1 & 0 \\ 0 & 1 \\  1 & 0 \\ \hline  
\end{array}
\stackrel{H}{\longrightarrow} \cdots\ , \\
$$
so the common central block here is bi-special, and it is the same 
as the common block $v$ of $H^n(01)$ and $H^n(10)$ 
of length $F_{n+2}-2$ in the proof of Lemma~\ref{lem-prop-stopbloc}.
Thus we have found the bi-special word of length $F_{n+2}-2$, and every
prefix and suffix of $v$ is left and right-special respectively.
The fact that these are the only bi-special words can be derived from 
the Rauzy graph for this Sturmian shift, see
\eg \cite[Sec. 1]{arnoux-rauzy}.
In their notation, there is a bi-special word of length $k$
if the two special nodes in the Rauzy graph coincide: $D_k = G_k$.
The lengths of the two ``buckles'' of non-special nodes between $D_k = G_k$
are two consecutive Fibonacci numbers minus one, as follows from the 
continued fraction expansion 
$$
\ga=1+\frac1{1+\frac1{1+\ddots}}.
$$
Therefore, the complexity satisfies
$$
k+1 = p(k) = \#\{ \text{nodes of Rauzy graph of order } k\}
= F_n-1 + F_{n-1}-1 + 1,
$$
so indeed only the numbers $k = F_{n+1}-2$ can be the lengths of bi-special 
words.
\end{proof}

\subsection{Iterations of the renormalization operator}

The renormalization operator  for potentials can be rewritten under as 
(recall the definition of $F_{n}^{a}$, $a=0,1$, from \eqref{eq:Fiboa})
\begin{equation}\label{equ-def-cr}
\CR V|_{[a]} = \sum_{j=0}^{F^{a}_{1}-1} V \circ \sigma^j \circ H|_{[a]}.
\end{equation}
This general formula may be extended to other substitutions and leads to an expression for $\CR^{n}V$.  The main result here is Lemma~\ref{lem:RkV}, where we show that
\begin{equation}\label{eq:RnV}
(\CR^n V)(x) = \sum_{j=0}^{F_{n^{*}}-1} V \circ \sigma^j \circ H^n(x),
\end{equation}
where
\begin{equation}\label{eq:n*}
n^* = \begin{cases}
n+1 & \text{ if } x \in [0],\\
n & \text{ if } x \in [1].
\end{cases}
\end{equation}
The substitution $H$ solves a renormalization equation of the form \eqref{equ1-renorm}. If $x=0x_{1}\ldots$, then $H(x)=01H(x_{1})\ldots$ and $\s^{2}\circ H(x)=H\circ \s(x)$. If $x=1x_{1}\ldots$ then we simply have 
$\s\circ H(x)=H\circ \s(x)$. The renormalization equation is thus more complicated than for the constant length case. We need an expression for iterations of $H$ and $\s$. 
\begin{lemma}\label{lem:commute_shift_H}
Given $k \ge 0$ and $a=0,1$, let 
$w = w_1w_2\dots w_{F^{a}_{k}} = H^k(a)$.
Then for every $0 \le i < F^{a}_{k}$ we have
$$
H \circ \sigma^i \circ H^k|_{[a]} = \sigma^{| H(w_1\dots w_i)|} \circ H^{k+1}|_{[a]}.
$$
\end{lemma}
\begin{proof}
For $k = 0$ this is true by default and for $k= 1$, this is precisely
what is done in the paragraph before the lemma.
Let us continue by induction, assuming that the statement is true for $k$.
Then $\sigma^i$ removes the first $i$ symbols of $w = H^k(a)$,
which otherwise, under $H$, would be extended to a word of
length $|H(w_1\dots w_i)|$. We need this number of shifts
to remove $H(w_1\dots w_i)$ from $H([w]) = H^{k+1}([a])$.
\end{proof}


\begin{lemma}\label{lem:RkV}
For every $k \ge 0$ and $a=0,1$, we have
$$
\CR^kV|_{[a]} = S_{F^{a}_{k}}V \circ H^k|_{[a]},
$$
where $S_nV = \sum_{i=0}^{n-1} V \circ \sigma^i$ denotes the $n$-th ergodic sum.
\end{lemma}

\begin{proof}
For $k = 0$ this is true by default. For $k = 1$, this follows by the definition of the renormalization operator $\CR$.
Let us continue by induction, assuming that the statement is true for $k$.
Write $w = H^k(a)$ and $t_i = |H(w_i)| = F_{w_i}$.
Then
\begin{eqnarray*}
\CR^{k+1}V|_{[a]} &=& (\CR V) \circ S_{F^a_k}V \circ H^k|_{[a]} \hskip 1cm \text{\small (Induction assumption)}\\
 &=& \sum_{i=0}^{F^{a}_{k}-1}  \left( \sum_{j=0}^{t_i-1} V \circ \sigma^j \circ H \right)  \sigma^i\circ H^k|_{[a]}\hskip 1cm \text{\small (by formula \eqref{equ-def-cr})}\\
 &=& \sum_{i=0}^{F^{a}_{k}-1} \left( \sum_{j=0}^{t_i-1} V \circ \sigma^{j+|H(w_1\dots w_i)|} \circ H \right) \circ H^k|_{[a]}  \hskip 1cm\text{\small (by Lemma~\ref{lem:commute_shift_H})}\\
 &=& \sum_{l=0}^{F^{a}_{k+1}-1} V \circ \sigma^l \circ H^{k+1}|_{[a]}, 
\end{eqnarray*}
as required.
\end{proof}



\subsection{Special words are sources of accidents}
Overlaps of  $\rho$ with itself are strongly related to bi-special words. They are of prime importance to determine the fixed points of $\CR$ and their 
stable leaves, see \eg formula \eqref{equ-crkV} below. Dynamically, they correspond to what we call {\em accident} in the time-evolution of the distance between the orbit and $\K$.

For most $x$ close to $\K$, $d(\s(x),\K) = 2d(x,\K)$, but
the variation of $d(\s^{j}(x),\K)$ is not always monotone with respect to $j$. When it decreases, it generates an accident:
\begin{definition}\label{def-accident}
Let $x\in\S$ and $d(x,\K)=2^{-n}$. If $d(\s(x),\K)\le 2^{-n}$, we say that we have an {\em accident} at $\s(x)$. 
If there is an accident at $\s^{j}(x)$, then we shall simply say we have an accident at $j$. 
\end{definition}

The next lemma allows us to detect accidents. 
\begin{lemma}\label{lem-accident-bispecial}
Let $x=x_{0}x_{1}\ldots$ coincide with some $y\in\K$ for $d$ digits. Assume that the first accident occurs at $b$. Then $x_{b}\ldots x_{d-1}$ is a bi-special word in $\CA_{\K}$. Moreover, the word $x_{0}\ldots x_{d-1}$ is not right-special. 
\end{lemma}
\begin{proof}
By definition of accident, there exists $y$ and $y'$ in $\K$ such that $d(x,\K)=d(x,y)$ and $d(\s^{b}(x),\K)=d(\s^{b}(x),y')$. 

\begin{figure}[htbp]
\begin{center}
\includegraphics[scale=0.7]{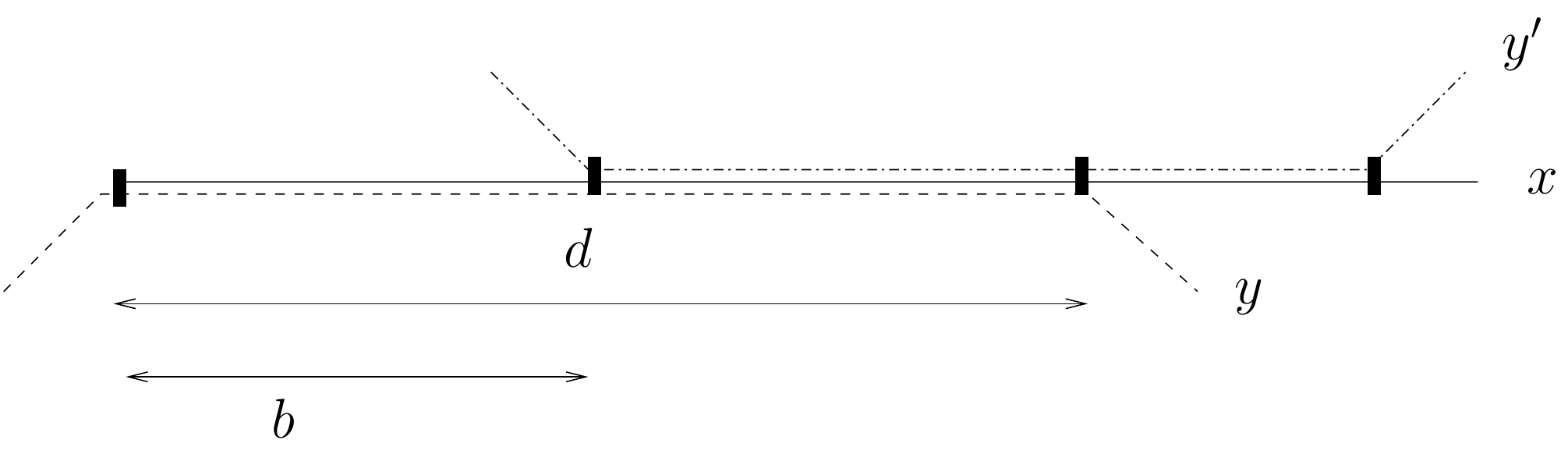}
\caption{Accident and bi-special words}
\label{fig-accident}
\end{center}
\end{figure}
Figure~\ref{fig-accident} shows that the word $x_{b}\ldots x_{d-1}$ is bi-special because its two extensions $y$ and $y'$ in $\K$ 
have different suffix and prefix for this word. 

It remains to prove that $x_{0}\ldots x_{d-1}$ is not right-special. If it was, then $x_{0}\ldots x_{d-1}x_{d}=y_{0}\ldots y_{d-1}\ol{y_{d}}$ 
would a $\K$-admissible word, thus $d(x,\K)\le 2^{-(d+1)}\neq 2^{-d}$. 
\end{proof}


\section{Proof of Theorem~\ref{theo-fixedpoint}}\label{sec-prooftheofix}
\subsection{Control of the accidents under iterations of $\CR$}
Next we compute $\CR^{n}V$ and show that accidents do not crucially
perturb the Birkhoff sum involved. This will follow from Corollaries~\ref{coro-dHnK} and \ref{cor-coinc-hn-rho}.

Note that Lemma~\ref{lem-prop-stopbloc} shows that $H$ is one-to-one. The next proposition explains the relation between the attractor $\K$ and its image by $H$. 

\begin{proposition}\label{prop-K-H}
The subshift $\K$ is contained in $H(\K) \cup \s\circ H(\K)$. More precisely, if $[0]\cap \K \subset H(\K)$ and $x\in[1]\cap \K \subset \s\circ H(\K)$.  
\end{proposition}

\begin{proof}
First note that Lemma~\ref{lem-prop-stopbloc} shows that $H$ is one-to-one. We also recall that the word $11$ is forbidden in $\K$. Hence, each digit $1$
in $x=x_{0}x_{1}x_{2}\ldots\in \K$ is followed and preceded by a digit $0$ (unless the $1$ is in first position). 

Assuming $x_{0}=0$, we can unique split $x$ into blocks of the form 
$0$ and $01$.
In this splitting, we replace 
each single $0$ by $1$ and each pair $01$ by $0$. 
This produces a new word, say $y$, and by construction, $H(y)=x$. 
This operation is denoted by $H^{-1}$. 
It can be used on finite words too, provided that the last digit is $1$. 
If $x_{0}=1$, we repeat the above construction with $0x$, and $x=\s\circ H(y)$. 

It remains to prove that $y\in\K$. 
For every $x\in\K$, there is a sequence $k_{n}\to\8$ 
such that $\s^{k_{n}}(\rho)\to x$. Assume again that $x_0 = 0$.
Then we can find a sequence $l_n \sim k_n/\ga$ such that 
$H \circ \s^{l_n}(\rho)=\s^{k_{n}}(\rho)$.
Therefore $\lim_n \s^{l_n}(\rho) \in \K$, and this limit is indeed the sequence
$y$ that satisfies $H(y) = x$.
Finally, for $x_0 = 1$, we repeat the argument with $0x$.
\end{proof}


\begin{corollary}\label{coro-dHnK}
If $d(x,\K)=d(x,y)$ with $y\in \K$, then $d(H^{n}(x),\K)=d(H^{n}(x),H^{n}(y))$ for $n\ge 0$.
\end{corollary}
\begin{proof}
Write $x=wa$ and $y=w\ol a$ where $a$ is an unknown digit and $\ol a$ its opposite. Note that $H^{n}(x)$ starts with $0$ for any $n\ge 1$.
Assume that there is some $z\in \K$ such that $d(H(x),z)<d(H(x),H(y))$.

\paragraph{{\it Case 1:}} $x=w0\ldots$ and $y=w1\ldots$. Necessarily, $y=w10$. Therefore $H(x)=H(w)01\ldots$ and $H(y)=H(w)001\ldots$. By assumption, $z$ coincide with $H(x)$ longer than $H(y)$, which shows that $z$ starts as 
$z=H(w)01\ldots$
Consequently $H^{-1}(z)=w0\dots$ and this contradicts that $d(x,\K)=d(x,y)$. 

\paragraph{{\it Case 2:}} $x=w1b\ldots$ and $y=w0\ldots$. Then $H(x)=H(w)00\ldots$  (since $H(b)$ starts with $0$ regardless what $b$ is)
and $H(y)=H(w)01\ldots$ Again $z$ coincides with $H(x)$ longer than $H(y)$ and thus $z$ starts as $H(w)00$. The $0$ before last position is necessarily a single zero for the $H^{-1}$-procedure and thus $H^{-1}(z)$ coincide with  $x$ for longer than $y$. This is a contradiction.

Consequently for both cases we have shown $d(H(x),\K)=d(H(x),H(y))$. 
The result follows by induction. 
\end{proof}

We recall that by \eqref{eq:RnV}, $\CR^{n}V$ is given by a Birkhoff sum, of
$F_{n^*}$ where $n^* = n+1$ or $n$ as in \eqref{eq:n*}.
To compute $(\CR^{n}V)(x)$, we need an estimate for $d(\s^{j}(H^{n}(x)),\K)$, for $0 \le j\le F_{n+1}-1$ or $0 \le j\le F_{n}-1$. The key point is that no accident can occur for these $j$. This follows from the next lemma. 

\begin{lemma}\label{lem-no-accident}
The sequence $H^n(x)$ has no accident in the first $F_{n^*}$ entries.
\end{lemma}

\begin{proof}
We give the proof for $x \in [1]$, so $n^* = n$. The proof for $x \in [0]$ is analogous.
By Corollary~\ref{cor-coinc-hn-rho}, $H^n(x)$ coincides for at least $F_{n+2}-2$
digits with $\rho$.
If an accident happens in the first $F_n$ digits, say at entry $0 \le j < F_n$,
then by Lemma~\ref{lem-accident-bispecial}, a bi-special word starts
at $j$, which by Proposition~\ref{prop-bispecialfibo} is a suffix
of $\rho$ of length $F_m-2$ for some $m$.
Since we have an accident, $j+F_m-2 \ge F_{n+2}-1$, so $m > n+1$.


Hence $\rho_{0}\ldots\rho_{F_{n+2}-1}$ can be written as $BBB'$ where $B$ is the suffix of $\rho$ of length $j$ and $B'$ is a suffix of $\rho$ of length $\ge |B|/\ga$.
Clearly $B$ starts with $0$. We can split it uniquely into blocks $0$ and $01$, and $B$ fits an integer number of such blocks, because if
the final block would overlap with the second appearance of $B$, then $B$ 
would start with $1$, which it does not.

Therefore we can perform an inverse substitution $H^{-1}$, for  each block $B$ and also for $B'$ because we cal globally do $H^{-1}$ for $\rho_{0}\ldots\rho_{F_{n+2}-1}$. We  find $H^{-1}(BBB') = CCC'$ which  has the same characteristics.
Repeating this inverse iteration, we find that $\rho$ starts with $0101$, or with $00$, 
a contradiction.
\end{proof}

Let $N(x,n)$ be the integer such that $2^{-N(x,n)} = d(H^{n}(x),\K)$.
By the previous lemma $d(\s^{j}(H^{n}(x))\K)=2^{-(N(x,n)-j)}$
for every $j<F_{n^{*}}$. For the largest value $j = F_{N^*}$,
we have  $d(\s^{j}(H^{n}(x))\K) = 2^{-(T_{n}+F_{n+2}-2-F_{n^{*}})}$.
Therefore, if $g$ is the $\al$-density function for $V$, then we obtain
\begin{equation}\label{equ-crkV}
(\CR^{n}V)(x)=\sum_{j=0}^{F_{n^*}-1}\frac{g \circ \s^j \circ H^n(x)}{(N(x,n)-j)^{\al}} +o\left(\sum_{j=0}^{F_{n^*}-1}\frac{g \circ \s^j \circ H^n(x)}{(N(x,n)-j)^{\al}}\right).
\end{equation}

\subsection{Proof of Theorem~\ref{theo-fixedpoint}}
\subsubsection{$\wt V$ is a fixed point}
A simple computation shows that $\CR$ fixes $\wt V$ from
\eqref{eq:tildeV}. Assume $x\notin\K$ is such that $\wt\kappa_{0}(x)=n$ and $\wt\kappa_{1}(x)=m$ (see the definition of $\wt\kappa_{a}$ above  the statement of Theorem~\ref{theo-fixedpoint}).
Then, by Lemmas~\ref{lem-lengthHn} and \ref{lem-prop-stopbloc}, 
and the fact that $H(x)$ starts with $0$, we get
\begin{eqnarray*}
\wt\kappa_{0}(H(x)) = n+m+1\ & \qquad &\wt\kappa_{0}(\s \circ H(x))  =  n+m,\\
\wt\kappa_{1}(H(x))  =  n \qquad \qquad & \qquad & \wt\kappa_{1}(\s \circ H(x)) =  n.
\end{eqnarray*}
 \begin{itemize}
\item If $x$ starts with $0$, then $H(x)$ starts with $01$ and 
\begin{eqnarray*}
(\CR\wt V)(x)&=& \wt V(H(x))+\wt V\circ \s(H(x))\\
&=& \log\left(\frac{(n+m+1)+\frac1\ga n+\ga}{(n+m+1)+\frac1\ga n+\ga-1}\right)+\log\left(\frac{\ga(n+m)+ n+\ga^{2}}{\ga(n+m)+ n+\ga^{2}-1}\right)\\
&=& \log\left(\frac{(n+m+1)+\frac1\ga n+\ga}{n+m+\frac1\ga n+\ga}\right)+\log\left(\frac{n+m+\frac1\ga n+\ga}{n+m+\frac1\ga n+\ga-\frac1\ga}\right)\\
&=& \log\left(\frac{n+m+1+\frac1\ga n+\ga}{n+m+\frac1\ga n+\ga-\frac1\ga}\right)\\
&=& \log\left(\frac{n(1+\frac1\ga)+m+\ga+1}{n(1+\frac1\ga)+m+\ga-\frac1\ga}\right)\qquad\qquad  \text{\small since } \ga^{2}=\ga+1\\
&=& \log\left(\frac{\ga n+m+\ga^{2}}{\ga n+m+\ga(\ga-1)}\right)=\log\left(\frac{n+\frac1\ga m+\ga}{n+\frac1\ga m+\ga-1}\right)=\wt V(x).
\end{eqnarray*}
 
\item If $x$ starts with $1$, then $H(x)$ starts with $0$ and 
\begin{eqnarray*}
(\CR\wt V)(x)&=& \wt V(H(x))=\log\left(\frac{(n+m+1)+ \frac1\ga n+\ga}{(n+m+1)+\frac1\ga n+\ga-1}\right)\\
&=&
\log\left(\frac{\ga(n+m+1)+ n+\ga^{2}}{\ga(n+m+1)+ n+\ga^{2}-\ga}\right)\\
&=&
\log\left(\frac{n(\ga+1)+\ga m+\ga+\ga^{2}}{n(\ga+1)+\ga m+\ga^{2}}\right)\\
&=& \log\left(\frac{\ga^{2}n+\ga m+\ga^{3}}{\ga^{2}n+\ga m+\ga^{2}}\right)\\
&=&
\log\left(\frac{\ga n+m+\ga^{2}}{\ga n+m+\ga}\right)
=\log\left(\frac{\ga n+m+\ga^{2}}{\ga n+m+\ga^{2}-1}\right) =\wt V(x).
\end{eqnarray*}
\end{itemize}

\subsubsection{A Toeplitz summation}
Next, we consider $V\in\CX_{1}$. We shall see in the proof that the convergence for these $V$ automatically implies that for $V\in\CX_{\al}$ with $\al\ne1$, the same computations work but with either too light or too heavy denominators. 

In the same spirit, we shall see that $\disp \sum_{j=0}^{F_{n^*}-1}\frac{g \circ \s^j \circ H^n(x)}{N(x,n)-j}$ actually converges, which immediately yields that  $\disp o\left(\sum_{j=0}^{F_{n^*}-1}\frac{g \circ \s^j \circ H^n(x)}{N(x,n)-j}\right)$  converges to $0$. We can thus forget any term of that kind. 

Recall that by Lemmas~\ref{lem-lengthHn} and \ref{lem-prop-stopbloc}
and Corollary~\ref{coro-dHnK} we have
$$
N(x,n) := \log_2 d(H^n(x), \K) = T_n+F_{n+2}-2 \quad \text{ for } \
T_{n}:=F_{n+1}\wt\kappa_{0}(x)+F_{n}\wt\kappa_{1}(x).
$$
We thus have to compute the limit of 
$$
\sum_{j=0}^{F_{n^*}-1}\frac{g \circ \s^j \circ H^n(x)}{F_{n+1}\wt\kappa_{0}(x)+F_{n}\wt\kappa_{1}(x)+F_{n+2}-(j+2)}
$$
as $n^* \to +\8$, where $n^{*} = n+1$ if $x \in [0]$ and $n^* = n$ if $x\in[1]$. Moreover $g$ is a non-negative continuous function, hence uniformly continuous. Furthermore, for $y \in \K$ closest to $x$,
the point $\s^{F_{n^{*}}} \circ H^{n}(x)$ coincides with $\s^{F_{n^{*}}} \circ H^{n}(y)$ for at least $F_{n}-2$ digits.
There exists a sequence  $\eps_{n}\downarrow 0$ such that 
$$\left|g\circ \s^{j}(H^{n}(x))-g\circ \s^{j}(H^{n}(y))\right|\le \eps_{n},$$
for every $j\le F_{n^{*}}-1$.

Finally, Binet's formula $F_{n+1}-\ga F_n = \sqrt{5} \ga^{-(n-1)}$ shows that 
\begin{eqnarray*}
F_{n+1}\wt\kappa_{0}(x)+F_{n}\wt\kappa_{1}(x)\!&+&\!F_{n+2}-(j+2)\\
&=& F_{n+1}(\wt\kappa_{0}(x)+\frac1\ga\wt\kappa_{1}(x)+\ga-\frac{j}{F_{n+1}})(1+\eps'_{n}) \\
&=& F_{n}(\ga\wt\kappa_{0}(x)+\wt\kappa_{1}(x)+\ga^{2}-\frac{j}{F_{n}})(1+\eps''_{n}),
\end{eqnarray*}
where $\eps'_{n}$ and $\eps''_{n}$ tend to $0$ as $n \to +\8$. 

Combining $\eps$, $\eps'$ and $\eps''$ in a single $o(1)$, 
we can rewrite the above equalities to 
\begin{equation}\label{equ1-toeplitz}
 \sum_{j=0}^{F_{n^*}-1}\frac{g \circ \s^j \circ H^n(x)}{F_{n+1}\wt\kappa_{0}(x)+F_{n}\wt\kappa_{1}(x)+F_{n+2}-(j+2)}=
\frac{1+o(1)}{F_{n^{*}}} \sum_{j=0}^{F_{n^{*}}}\frac{g\circ\s^{j}\circ H^{n}(y)}{X_{n}-\frac{j}{F_{n^{*}}}},
\end{equation}
where $X_{n}=\wt\kappa_{0}(x)+\frac1\ga\wt\kappa_{1}(x)+\ga$ if $x \in [0]$ 
and $X_{n}=\ga\wt\kappa_{0}(x)+\wt\kappa_{1}(x)+\ga^{2}$ if $x\in[1]$. 

\subsubsection{Convergence of the weighted sum in \eqref{equ1-toeplitz}}
The reader can verify that we are here considering a Toeplitz summation method, with a regular matrix (see \cite[ Definition 7.5]{kuipers-niederreiter}  and \cite{weber-crass}), up to a renormalization factor, which is the limit of 
$$
\frac1{F_{n^{*}}}\sum_{j=0}^{F_{n^{*}}}\frac1{X_{n}-\frac{j}{F_{n^{*}}}}.
$$
This expression is a Riemann sum, and converges to $\wt V(x)$
as $n^* \to \8$. 

From \cite{weber-crass} we have that the M\"uller (see \cite{muller}) criterion applies, and we get that for $\mu_{\K}$-almost every $z\in\K$, 
\begin{equation}
\label{equ-cv-toepli}
\lim_{n\to+\8}\frac1{F_{n^{*}}}\sum_{j=0}^{F_{n^{*}}}\frac{g\circ \s^{j}(z)}{X_{n}-\frac{j}{F_{n^{*}}}}=\wt V(x)\int g\,d\mu_{\K}.
\end{equation}

Nevertheless, our expression in \eqref{equ1-toeplitz} is different, because the point $z$ we are considering is $H^{n}(y)$ and thus depends on $n$.
A priori, this may generates fluctuations in the convergence, but
we prove here that this is not the case.

The main argument is that $(\K,\s)$ is uniquely ergodic. We claim that this implies that the convergence in \eqref{equ-cv-toepli} is uniform in $z$. 

Indeed, if it is not uniform, we can find $\eps > 0$ and 
a sequence of $z_{n}$ such that for every $n$, 
$\disp |\frac1{F_{n^{*}}}\sum_{j=0}^{F_{n^{*}}}\frac{g\circ \s^{j}(z_{n})}{X_{n}-\frac{j}{F_{n^{*}}}} - \wt V(x)\int g\,d\mu_{\K}| > \eps$ 
for every $n$. 
Then any accumulation point $\mu_{\8}$ of the family of measures 
$$
\mu_{n}:=\frac1{F_{n^{*}}}\sum_{j=0}^{F_{n^{*}}}\frac{1}{X_{n}-\frac{j}{F_{n^{*}}}}\delta_{\s^{j}(z_{n})}
$$
is $\s$-invariant (because $F_{n^{*}}\to+\8$), supported on $\K$, and 
$\int g\,d\mu_{\8}\neq \int g\,d\mu_{\K}$. This would contradict the unique ergodicity for $(\K, \s)$. 

Therefore, the convergence in \eqref{equ-cv-toepli} is uniform in $z$ and this shows that 
$$
\frac1{F_{n^{*}}}\sum_{j=0}^{F_{n^{*}}}\frac{g\circ \s^{j}(H^{n}(y))}{X_{n}-\frac{j}{F_{n^{*}}}} \to \wt V(x) \cdot \int g\,d\mu_{\K}.
$$ 
This finishes the proof of Theorem~\ref{theo-fixedpoint}.

\section{Proof of Theorem~\ref{theo-pt}}\label{sec-proffthpt}

\subsection{The case \boldmath $-\log \frac{n+1}{n}$ \unboldmath}
\label{subsec-logcase}
We first consider the potential $\varphi(x) = -\log \frac{n+1}{n}$ if $d(x, \K) = 2^{-n}$, leaving the general potential in $\CX_1$ for later.

\subsubsection{Strategy, local equilibria}
Fix some cylinder $J$ such that the associated word, say $\omega_{J}$, does not appear in $\rho$ (as \eg 11). We follow the induction method presented in \cite{leplaideur-synth}. Let $\tau$ be the first return time into $J$ (possibly $\tau(x)=+\8$), and consider the family of transfer operators 
\begin{eqnarray*}
\CL_{Z,\be}:\psi&\mapsto& \CL_{Z,\be}(\psi)\\
x&\mapsto& \CL_{Z,\be}(\psi)(x):=\sum_{n=1}^{+\8}\sum_{\stackrel{y\in J\ \tau(y)=n}{ \s^{n}(y)=x}}e^{\be \cdot (S_{n}\varphi)(y)-nZ}\psi(y),
\end{eqnarray*}
which acts on the set of continuous functions $\psi:J\to\R$.  
Following \cite[Proposition 1]{leplaideur-synth}, for each $\be$ there exists $Z_{c}(\be)$ such that $\CL_{Z,\be}$ is well defined for every $Z>Z_{c}(\be)$. 
By \cite[Theorem 1]{leplaideur-synth}, $Z_{c}(\be)\ge 0$ because the pressure of the dotted system 
(which in the terminology of \cite{leplaideur-synth} is the system restricted to the trajectories that avoid $J$)
is larger (or equal) than the pressure of $\K$ which is zero. 

We shall prove
\begin{proposition}
\label{prop-spectral&zc}
There exists $\be_{0}$ such that $\CL_{0,\be}(\BBone_{J})(x)<1$ for every $\be>\be_{0}$ and $x\in J$.
\end{proposition}
We claim that if Proposition~\ref{prop-spectral&zc} holds, then
\cite[Theorem 4]{leplaideur-synth} proves that $\CP(\be)=0$ for every $\be>\be_{0}$, and $\mu_{\K}$ is the unique equilibrium state for $\be \varphi$. 

To summarize \cite{leplaideur-synth} (and adapt it to our context), the 
pressure function satisfies (see Figure~\ref{fig-graphes-press}\footnote{We will see that $\CL_{0,\be}(\BBone_{J})(x)$ is a constant function on $J$.}),
$$
Z_{c}(\be) \le \CP(\be) \le \max(\log(\CL_{0,\be}(\BBone_{J})),0).
$$

As long as $\CP(\be)>0$, there is a unique equilibrium state and it has full support. In particular this shows that the construction does not depend on the choice of $J$. If Proposition~\ref{prop-spectral&zc} holds, then 
$$Z_{c}(\be)=\CP(\be)=\max(\log(\CL_{0,\be}(\BBone_{J})),0)=0,
\text{ for } \be>\be_{c}$$
and $\mu_{\K}$ is the unique equilibrium state because $\CL_{0,\be}(\BBone_{J})<1$. 

\begin{figure}[htbp]
\begin{center}
\includegraphics[scale=.5]{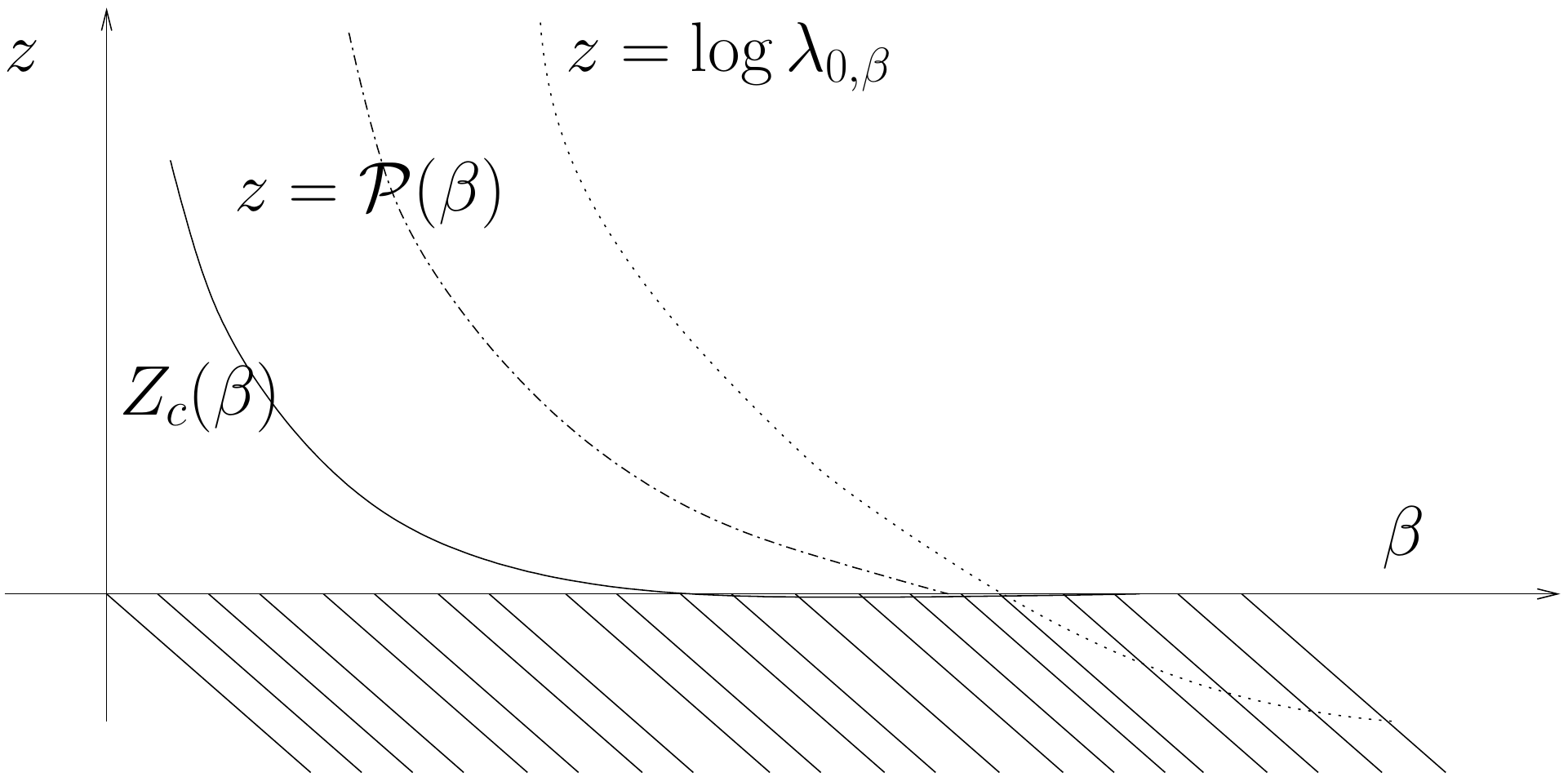}
\caption{The Pressure between $Z_{c}(\be)$ and $\log\l_{0,\be}:=\CL_{0,\be}(\BBone_{J})$}
\label{fig-graphes-press}
\end{center}
\end{figure}

\subsubsection{Proof of Proposition~\ref{prop-spectral&zc}-Step 1}\label{subsubsec-upper=series}
We reduce the problem to the computation of a series depending on $\be$. 
Note that $\varphi(x)$ only depends on the distance from $x$ to $\K$. This shows that if $x, x' \in J$ and $y, y' \in J$ are such that 
$$y=\omega x\ , \quad  y'=\omega x',$$
with $\omega\in \{0,1\}^{n}$, $\tau(y)=\tau(y')=n$, then 
$$(S_{n}\varphi)(y)=(S_{n}\varphi)(y').$$
In other words, $\CL_{Z,\be}(\BBone_{J})$ is a constant function, and then equal to the spectral radius $\l_{Z,\be}$ of $\CL_{Z,\be}$.  

Consequently, to compute $\l_{Z,\be}$, it suffices to compute the sum of all 
$\disp e^{\be \cdot (S_{n}\varphi)(\omega)-nZ}$, where $\omega$ is a word of length $n+|\omega_{J}|$, starting and finishing as $\omega_{J}$. Such a word $\omega$ can also be seen as a path of length $n$ starting from $J$ and returning (for the first time) to $J$ at time $n$. 

We split such a path in several sub-paths. We fix an integer $N$ and say that the path is {\em free} at time  $k$ if $\omega_{k}\ldots \omega_{n-1}\omega_{J}$ is at distance  larger than $2^{-N}$ to $\K$. Otherwise, we say that we have an {\em excursion}. The path is thus split into intervals of free moments and excursions. 
We assume that $N$ is chosen so large that $0$ is a free moment. This also shows that for every $k\le n$, $d(\s^{k}(\omega\omega_{J}),\K)$ is determined by $\omega_{k}\ldots \omega_{n-1}$. 

If $k$ is a free time, $\disp \varphi(\s^{k}(\omega\omega_{J}))\le A_{N}:=-\log\left(1+\frac 1N\right)$. Denote by $k_{0}$ the maximal integer such that 
$k$ is a free time for every $k\le k_{0}$. Then $S_{k_{0}+1}\varphi\le (k_{0}+1)A_{N}$ and there are fewer than $2^{k_{0}+1}$ such prefixes of length $k_0+1$. 

Now, assume that every $j$  for $k_{0}+1\le j\le k_{0}+k_{1}$ is an excursion time, and assume that $k_{1}$ is the maximal integer with this property. To the contribution $(S_{k_{0}+1}\varphi)(\omega\omega_{J})$ we must add the contribution $(S_{k_{1}}\varphi)(\s^{k_{0}+1}(\omega\omega_{J}))$ of the excursion. 
Then we have a new interval of free times, and so on. 
This means that we can compute $\CL_{0,\be}(\BBone_{J})$ by gluing together paths with the same decompositions  of free times and excursion times. If we denote by $C_{E}$ the total contribution of all paths with exactly one excursion (and only starting at the beginning of the excursion), then we have 
\begin{equation}
\label{equ1-upperboundCL0}
\l_{0,\be}=\CL_{0,\be}(\BBone_{J}) \le \sum_{k=1}^{+\8}\left(\sum_{k_{0}=0}^{+\8}e^{(k_{0}+1)(\be A_{N}+\log2)}\right)^{k+1}C_{E}^k.
\end{equation}
The sum in $k$ accounts for $k+1$ intervals of free moment with $k$ intervals of excursions times between them. The sum in $k_{0}$ accounts for the possible length $k_{0}+1$ for an interval of free times. 
These events are maybe not independent but the sum in \eqref{equ1-upperboundCL0} includes all paths, possible or not, and therefore yields an upper bound. 

The integer $N$ is fixed, and we can take $\be$ so large that $\be A_{N}<-\log2$. This shows that the sum in $k_{0}$ in \eqref{equ1-upperboundCL0}  converges and is as close to $0$ as we 
want if $\be$ is taken sufficiently large. 

To prove Proposition~\ref{prop-spectral&zc}, it is thus sufficient to prove that $C_{E}$ can be made as small as we want if $\be$ increases. 

\subsubsection{Proof of Proposition~\ref{prop-spectral&zc}-Step 2}\label{subsubsec-splitCE}
We split excursions according to their number of accidents, see Definition~\ref{def-accident}. 
Let $x$ be a point at a beginning of an excursion. 
 
Let $B_{0}:=0=b_{0}$, $B_1 := b_1>b_{0},\ B_2 := b_1+b_2>b_{1},\ 
B_3 := b_1+b_2+b_3, \dots ,
B_M := b_1+b_2+ \dots + b_M$
be the times of accidents in the excursion.
There is $y_0 \in \K$ such that $x$ shadows $y_0$ at the beginning
of the excursion, say for $d_0$ iterates.
Let $y_i \in \K$, $i = 1, \dots, M$, be the points that $x$ starts to shadow 
at the $i$-th accident, for $d_i$ iterates.

\begin{figure}[htbp]
\begin{center}
\includegraphics[scale=0.5]{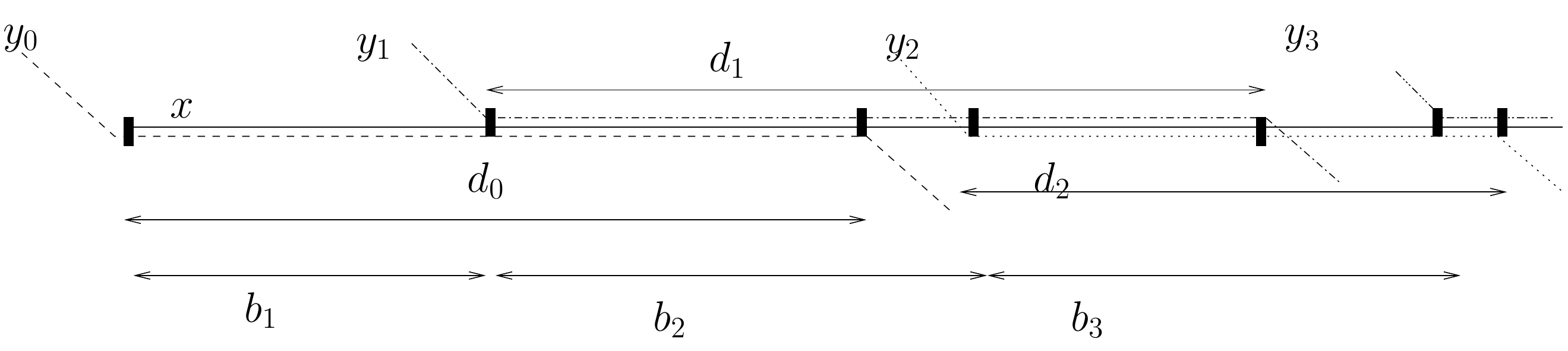}
\caption{Accidents during an excursion.}
\label{fig-excursion}
\end{center}
\end{figure}

Then by Lemma~\ref{lem-accident-bispecial}, $x_{b_{i+1}}\ldots x_{d_{i}}$ is bi-special and by Proposition~\ref{prop-bispecialfibo}, $d_i-b_{i+1} = F_{n_{i+1}}-2$ for some $n_{i+1}$. 

\begin{remark}\label{rem-yispe}
We emphasize that the first $d_i$ entries of $y_i$
do not form a special word. Indeed, it is neither right-special (due to
Lemma~\ref{lem-accident-bispecial}) nor left-special, because otherwise there would be an accident earlier. 
\hfill $\blacksquare$\end{remark}

If there are $M+1$ accidents (counting the first as 0), the ergodic sums  for $\varphi$ are 
\begin{eqnarray*}
(S_{b_{i+1}}\varphi)(\sigma^{B_{i}} (x) ) &=& 
 \sum_{k=0}^{b_{i+1}-1}\varphi \circ \sigma^{B_{i}+k} (x)\\
 &=& 
 \sum_{k=0}^{b_{i+1}-1} -\log \frac{d_{i}+1-k}{d_{i}-k} \\
&=& -\log \frac{d_{i}+1}{d_{i}+1-b_{i+1}} = -\log(1+  \frac{b_{i+1}}{d_i+1-b_{i+1}}),
\end{eqnarray*}
for $0\le i\le M-1$, while the ergodic sum of the tail of the excursion is

\begin{equation}
\label{equ-estiEM+1}
(S_{d_M}\varphi)(\sigma^{B_M} (x)) =  
\sum_{k=0}^{d_M-1} \varphi \circ \sigma^{B_M+k}( x) =
-\log \frac{d_M+1}{N+1}.
\end{equation}

We set $\disp \e_{i} := e^{\be \cdot (S_{b_{i}}\varphi)(\sigma^{B_{i-1}} (x) )}$ for $i=1\ldots M$ and $\disp \e_{M+1} := e^{\be \cdot (S_{d_M}\varphi)(\sigma^{B_M} (x))}$. 
Computing $C_{E}$, we can order excursions according to their number of accidents ($M+1$) and then according to the contribution of each accident. Let $E_{i}$ stand for the total contribution of all possible $\e_{i}$'s 
between accidents $i-1$ and $i$. 
Then
\begin{equation}\label{equ-defCE}
C_{E}=\sum_{M=0}^{+\8}\prod_{i=1}^{M+1}E_{i}.
\end{equation}

\subsubsection{Proof of Proposition~\ref{prop-spectral&zc}-Step 3}\label{subsubsec-computCE}
Let us now find an upper bound for $E_{i}$. 
By definition, $E_{i}$ is the sum over the possible $d_{i-1}$ and $b_{i}$ of $\e_{i}$. 

Recall $d_{i-1}-b_{i}=F_{n_{i}}-2$, so $b_{i}$ and $F_{n_{i}}$
determine $d_{i-1}$. 
The key idea is that $F_{n_{i}}$ and $F_{n_{i+1}}$ determine the possible values of $b_{i}$. 
This implies that $E_{i}$ can be written as an expression over the $F_{n_{i}}$ and $F_{n_{i+1}}$. 

\medskip
$\bullet$ For $2\le i\le M$ each $\e_{i}$ depends on $F_{n_{i}}$ and $b_{i}$. 
Let us show that  for $2\le i\le M$, $b_{i}$ depends on $n_{i}$ and $n_{i-1}$. 
Indeed, the sequence $y_{i} \in \K$ coincides for
$F_{n_{i}}-2$ initial symbols with $\rho$, and from entry $b_{i+1}$ has another
$d_{i} - b_{i+1} = F_{n_{i+1}}-2$ symbols in common with the head of $\rho$,
but differs from $x_{B_{i}+ d_i}$ at entry $d_i$, see Figure~\ref{fig-excursion}.
Thus we need to find all the values of $d_i > F_{n_i}-2$ such that
$\rho_0 \dots \rho_{d_i-1}$ ends the bi-special word $\rho_0 \dots \rho_{F_{n_{i+1}-3}}$ but is itself not bi-special.
The possible starting positions of this appearance
of $\rho_0 \dots \rho_{F_{n_{i+1}-3}}$ are the required numbers $b_{i+1}$.

\begin{lemma}\label{lem:bij}
Let us denote by $b_{i+1}(j)$, $j \ge 1$, the $j$-th value that
 $b_{i+1}$ can assume. Then 
\begin{equation}\label{equ-estibj}
b_{i+1}(j)\ge\max(F_{n_{i}}-F_{n_{i+1}},F_{n_{i}-1})+j F_{n_{i+1}-2}.
\end{equation}
\end{lemma}

This will allow us to find an upper bound for $E_{i}$ for $1\le i\le M-1$
later in this section. 

\begin{proof}
We abbreviate the bi-special words $L_k = \rho_0 \dots \rho_{F_k-3}$
for $k \ge 4$. 
For the smallest value $d_i \ge F_{n_i}-2$ so that
$\rho_0 \dots \rho_{d_i-1}$ ends in (but is not identical to) a block $L_{n_{i+1}}$,
this block starts at entry:
$$
b_{i+1}(0) =
\begin{cases}
F_{n_i}-F_{n_{i+1}} &\text{ if } n_{i+1} < n_i \text{ and } n_i-n_{i+1} \text{ is even,}\\
F_{n_i}-F_{n_{i+1}-1} &\text{ if } n_{i+1} < n_i  \text{ and } n_i-n_{i+1} \text{ is odd,}\\
F_{n_{i+1}-1} &\text{ if } n_{i+1} \ge n_i.
\end{cases}
$$
However, if $n_{i+1} < n_i$ then $d_i = F_{n_i}-2$ and if 
 $n_{i+1} \ge n_i$ then $d_i = F_{n_{i+1}+1}-2$ in this case, and thus
$\rho_0 \dots \rho_{d_i-1}$ is right-special, contradicting
Lemma~\ref{lem-accident-bispecial}.
Therefore we need to wait for the next appearance of $L_{n_{i+1}}$.
For the Rauzy graph of the Fibonacci shift, the bi-special word $L_k$
is the single node connecting loops of length $F_{k-1}$ and $F_{k-2}$,
see \cite[Section 1]{arnoux-rauzy}. Therefore the gap between two 
appearances of $L_k$ is always $F_{k-2}$ or $F_{k-1}$.
This gives $b_{i+1}(j+1) \ge b_{i+1}(j) + F_{n_{i+1}-2}$ for all $j \ge 0$
and \eqref{equ-estibj} follows.
\end{proof}

$\bullet$ For $i=1$, formula \eqref{equ-estibj} can be applied, if we introduce the quantity $n_{0}$, coinciding with the overlap of the end of the previous 
``fictitious'' word, say $y_{-1}$. The point is that $y_{0}$ is the ``beginning'' of the excursion, thus the first accident. Then, $F_{n_{0}}\le N$ and $F_{n_{1}}>N$, which yields $n_{0}<n_{1}$. Therefore $b_{1}=F_{n_{0}}-2+\disp \frac{j}\ga(F_{n_{1}}-2)$ with $j\ge 0$. 

$\bullet$ The estimation 
$$
E_{M+1} = \sum_{d \ge 1} e^{-\be \log(\frac{F_{N_m}+d}{N+1})}
= \sum_{d \ge 1} \left( \frac{F_{N_m}+d}{N+1} \right)^{-\be}
\le \frac{N+1}{\be-1} \left( \frac{F_{N_m}}{N+1} \right)^{1-\be}
$$ 
follows from by \eqref{equ-estiEM+1}, with $d_M=F_{n_M}+d$ and $d\ge 1$. 

\medskip
Recall that within excursions, all $F_{n_{j}} \ge N+1$
for all $j$, where $N$ can be chosen as large as we want.
By Binet's formula $F_n = \frac{1}{\sqrt{5}}(\ga^{n+1}-(-1/\ga)^{n+1})$, 
we can replace $F_{n}$ by $\disp\frac{\ga^{n+1}}{\sqrt 5}$, and  
we can also treat the quantities $-1$ as negligible compared to $\ga^{n_i}$.  
Therefore, assuming that $\beta > 1$, we obtain
\begin{eqnarray*}
E_{i} &=& \sum_{j \ge 1} 
e^{\disp-\beta  \log\left(1 +  \frac{ \max(F_{n_{i}}-F_{n_{i+1}},F_{n_{i}-1}) + j F_{n_{i+1}-2}}{F_{n_{i+1}}-1}\right)} \\
&\simeq& \sum_{j \ge 1} \left(1+\max(\ga^{n_{i} - n_{i+1}}-1,\ga^{n_{i}-n_{i+1}-1} )+ j/\ga^2\right)^{-\beta}\\
&\le& \frac{\ga^2}{\beta-1} \left(1+\max(\ga^{n_{i} - n_{i+1}}-1,\ga^{n_{i}-n_{i+1}-1} )\right)^{1-\beta}.
\end{eqnarray*} 
for $2\le i\le M$

Let $P\approx \disp \frac{\log\frac{N}{\sqrt5}}{\log\ga}$ be the largest integer $n$ such that $F_{n}\le N$. Then \eqref{equ-defCE} yields 
\begin{align}\label{eq:est0}
\nonumber C_{E}\le & \sum_{M=0}^{+\8}\left(\frac{\ga^2}{\be-1}\right)^{M}\frac{(N+1)}{\be-1} \ \cdot \\
& \sum_{\stackrel{n_{1},\ldots n_{M}> P}{n_{0}\le P}} E_{1}\left(1+\max(\ga^{n_{i} - n_{i+1}}-1,\ga^{n_{i}-n_{i+1}-1} )\right)^{1-\beta}\ga^{(P-n_{M})(\be-1)}.
\end{align}

\subsubsection{Proof of Proposition~\ref{prop-spectral&zc}-Step 4}\label{subsubsec-CEsmall}
We show that $C_{E}\to0$ as $\be\to+\8$.

\begin{proposition}
\label{prop-Cezero}
There exists $A=A(\be)\in (0,1)$ with $\lim_{\be\to+\8} A = 0$ such that 
$$
C_{E}\le 2P\, \frac{N+1}{\be-1}\sum_{n=1}^{+\8}\ga^{-n(\be-1)}\sum_{M=0}^{+\8}A^{M}\sum_{i=0}^{M}\frac{n^{i}}{i!}.
$$
\end{proposition}

Before proving this proposition, we show that is finishes the proof of Proposition~\ref{prop-spectral&zc}. 
The series has only positive terms. Clearly, $\disp \sum_{M=0}^{+\8}A^{M}\sum_{i=0}^{M}\frac{n^{i}}{i!}\le \frac1{1-A}e^{n}$, so the main sum converges if
$\ga^{\be-1}>e$. Thus Proposition~\ref{prop-Cezero} implies that $C_{E}\to 0$ as $\be\to+\8$. 

Therefore, inequality \eqref{equ1-upperboundCL0} shows that if $\be\to+\8$, 
then $\l_{0,\be}\to0$ too, and hence Proposition~\ref{prop-spectral&zc} is proved. 

The rest of this subsection is then devoted to the proof of Proposition~\ref{prop-Cezero}. 
\begin{lemma}
\label{lem-gamma-n}
Let $\eta$ and $y$ be positive real numbers. Then for every $n$, 
$$\int_{y}^{\8}x^{n}e^{-\eta (x-y)}dx=\sum_{j=0}^{n}\frac{n!}{j!}\frac{y^{j}}{\eta^{n+1-j}}.$$
\end{lemma}
\begin{proof}
Set $u_{n}:=\disp \int_{y}^{\8}x^{n}e^{-\eta (x-y)}dx$. 
Then
\begin{eqnarray*}
u_{n}&=& \int_{0}^{\8}(x+y)^{n}e^{-\eta x}dx\\
&=&\left[\frac{-1}{\eta}(x+y)^{n}e^{-\eta x}\right]_{0}^{\8}+\frac{n}\eta\int_{0}^{\8}(x+y)^{n-1}e^{-\eta x}\\
&=& \frac{y^{n}}\eta+\frac{n}\eta u_{n-1}.
\end{eqnarray*}
The formula follows by induction.
\end{proof}

Let $n$ be some positive integer and $\xi$ and $\zeta$ two positive real numbers. We consider a matrix $D_{n} = (d_{n,i,j})_{i=1, j=1}^{n+1, n}$ with $n+1$ rows and $n$ columns defined by 
$$
d_{n,i,j}:= \begin{cases}
 \frac{(j-1)!}{(i-1)!}\zeta^{j-i+1} & \text{ if } i \le j,\\[1mm]
 \frac\xi{j} & \text{ if }i=j+1,\\[1mm]
 0 &\text{ if }i>j+1.
\end{cases}
$$
or in other words:
$$
D_n=\left(\begin{array}{cc cc ccc}
0!\zeta & 1!\zeta^2 & 2!\zeta^3 & \ldots & (j-1)!\zeta^j & \ldots &   (n-1)!\zeta^{n} \\
\xi & \zeta & \ldots &  &  &  & (n-2)!\zeta^{n-1} \\
0 & \frac\xi2 & \zeta &  &  &  & \vdots \\
0 & 0 & \frac\xi3 & \ddots & \frac{(j-1)!}{(i-1)!}\zeta^{j-i+1} & &  \vdots\\
\vdots &  & 0 & \ddots & \ddots &  & \vdots\\
\vdots &  & & 0 & \frac{\xi}{j} & \zeta & \zeta^2 \\
0 &  &  &  & 0 & \frac{\xi}{n-1} & \zeta \\
0 & 0 & \ldots & \ldots & 0 & 0 & \frac{\xi}{n} 
\end{array}\right).
$$
We call $\w$ non-negative (and write $\w\succeq 0$) if all a entries of $\w$ are non-negative. This defines a partial ordering on vectors by 
$$\w' \succeq \w \iff \w'-\w\succeq 0.$$  

\begin{lemma}
\label{lem-matrix}
Assume $0<\zeta<1$ and 
set $K:=\frac1{1-\zeta}$. Then, for every $n$,
$$
D_n \cdot \left(\begin{array}{c}\disp\frac{K^{n-1}}{0!} \\[1mm]
\disp\frac{K^{n-1}}{1!} \\[1mm]
\disp\frac{K^{n-1}}{2!} \\[1mm]
\vdots \\[1mm]
\disp\frac{K^{n-1}}{(n-1)!}\end{array}\right)
\preceq
\left(\begin{array}{c}\disp\frac{K^{n}}{0!} \\[1mm]
\disp\frac{K^{n}}{1!} \\[1mm]
\disp\frac{K^{n}}{2!} \\[1mm]
\vdots \\[1mm]
\disp\frac{K^{n}}{n!}\end{array}\right).$$
\end{lemma}

\begin{proof}
This is just a computation. For the first row we get 
$$
\sum_{j=1}^{n}(j-1)!\zeta^{j}.\frac{K^{n-1}}{(j-1)!}\le K^{n-1}.\frac\zeta{1-\zeta}\le K^{n}.$$
For row $i>1$ we get 
$$\frac1{(i-1)}\frac{K^{n-1}}{(i-2)!}+\sum_{j=i}^{n}\frac{(j-1)!}{(i-1)!}\zeta^{j-i+1}\frac{K^{n-1}}{(j-1)!}=\frac{K^{n-1}}{(i-1)!}\left(1+\zeta+\zeta^{2}\ldots\right)\le \frac{K^{n}}{(i-1)!}.$$
\end{proof}

\begin{proposition}
\label{prop-calcul-matrix-majo}
Set $\zeta:=\disp\frac1{(\be-1)\log\ga}$ and $K=\frac{1}{1-\zeta}$. 
Consider $M$ integers $n_{1},\ldots n_{M}$, with $n_{M}> P$. Then, 
for every $M\ge 2$, 
$$
\sum_{n_1, \dots, n_{M-1} > P} \prod_{i=1}^{M-1} \left(1+\max(\ga^{n_{i} - n_{i+1}}-1,\ga^{n_{i}-n_{i+1}-1} ) \right)^{1-\beta}\le K^{M-1}\sum_{i=0}^{M-1}\frac{(n_{M}-P)^{i}}{i!}$$
\end{proposition}
\begin{proof}
Note that 
\begin{align*}
\sum_{n_1, \dots, n_{M-1} > P} & \prod_{i=1}^M \left(1+\max(\ga^{n_{i} - n_{i+1}}-1,\ga^{n_{i}-n_{i+1}-1} )\right)^{1-\beta}\\
= & \sum_{n_{M-1}=1}^{\8}\left(\ldots\left(\sum_{n_{2}=1}^{\8}\left(\sum_{n_{1}=1}^{\8}\right.\right.\right.
\left(1+\max(\ga^{n_{1} - n_{2}}-1,\ga^{n_{1}-n_{2}-1} ))^{1-\be}\right)\cdot\\
& \left.\left((1+\max(\ga^{n_{2} - n_{3}}-1,\ga^{n_{2}-n_{3}-1} ))^{1-\be}\right)\ldots\right) \\
& \left(1+\max(\ga^{n_{M-1} - n_{M}}-1,\ga^{n_{M-1}-n_{M}-1} \right)^{1-\be}.
\end{align*}
This means that we can proceed by induction. Now 
\begin{align*}
\sum_{n_{1}=P+1}^{\8} & (1+\max(\ga^{n_{1} - n_{2}}-1,\ga^{n_{1}-n_{2}-1} ))^{1-\be}\\
&\le\int_{P}^{n_{2}} (1+\ga^{x-n_{2}-1} ))^{1-\be}dx
+\int_{n_{2}}^{\8} (\ga^{x - n_{2}})^{1-\be}dx\\
&\le n_{2}-P+\int_{n_{2}}^{\8}e^{-(\be-1)\,(x-n_{2})\,\log\ga}\,dx\\
&= n_{2}-P+\int_{n_{2}}^{\8}e^{-\frac{x-n_{2}}{\zeta}}\,dx,
\end{align*}
because $\zeta=\frac1{(\be-1)\log\ga}$. This shows that the result holds for $M=2$.

Assuming that the sum for $M=p$ is of the form 
$\sum_{j=0}^{p-1}a_{j}(n_{p}-P)^{j}$, we compute the sum for $M=p+1$.
\begin{align*}
\sum_{n_{p}=P+1}^{\8} & \sum_{j=0}^{p-1}  a_{j}\frac{(n_{p}-P)^{j}}{(1+\max(\ga^{n_{p}-n_{p+1}}-1,\ga^{n_{p}-n_{p+1}-1} ))^{\be-1}} \\ 
&\le\ \sum_{j}a_{j}\int_{P}^{n_{p+1}}\frac{(x-P)^{j}}{(1+\ga^{x-n_{p+1}-1})^{\be-1}}\,dx+
\sum_{j}a_{j}\int_{n_{p+1}}^{\8}\frac{(x-P)^{j}}{(\ga^{x-n_{p+1}})^{\be-1}}\,dx\\
&\le \sum_{j}\frac{a_{j}(n_{p+1}-P)^{j+1}}{(j+1)}+\int_{n_{p+1}}^{\8}(x-P)^{j}e^{-\frac{x-n_{p+1}}{\zeta}}\,dx.
\end{align*}
Set $\disp \w\cdot\w'=\sum w_{i}w'_{i},$
for vectors $\w=(w_{1},\ldots, w_{p+1})$ and $\w'=(w'_{1},\ldots, w'_{p+1})$. 
Lemma~\ref{lem-gamma-n} yields 
 \begin{align*}
\sum_{n_{p}=P+1}^{\8} & \sum_{j=0}^{p-1} a_{j}\frac{(n_{p}-P)^{j}}{(1+\max(\ga^{n_{p}-n_{p+1}}-1,\ga^{n_{p}-n_{p+1}-1}))^{\be-1}}\\
&\le \sum_{j}\frac{a_{j}}{(j+1)}(n_{p+1}-P)^{j+1} + 
\sum_{i=0}^{j}\frac{j!}{i!}\zeta^{j-i+1}(n_{p+1}-P)^{i}\\
&\le D_{p}\left(\begin{array}{c}a_0 \\a_1 \\ \vdots \\ a_{p-1}
\end{array}\right) \cdot 
\left(\begin{array}{c}1 \\n_{p+1} \\\vdots \\n_{p+1}^{p}\end{array}\right).
\end{align*}
Lemma~\ref{lem-matrix} concludes the proof of the induction. 
\end{proof}

\begin{proof}[Proof of Proposition~\ref{prop-Cezero}]
We have just proven that 
\begin{align*}
\sum_{n_{1},\ldots n_{M}>P} & \left(1+\max(\ga^{n_{i} - n_{i+1}}-1,\ga^{n_{i}-n_{i+1}-1} )\right)^{1-\beta}\ga^{(P-n_{M})(\be-1)} \\
& \le K^{M-1}\sum_{n_{M}=P+1}^{+\8}\sum_{j=0}^{M-1}\frac{(n_{M}-P)^{j}}{j!}\ga^{(n_{M}-P)(\be-1)}.
\end{align*}

It remains to sum over $n_{0}$. Note that in that case, there are only $P$ terms of the form 
$\disp \sum_{j=0}^{+\8}\frac1{\left(1+\ga^{n_{0}-n_{1}-2}+\frac{j}{\ga}\right)^{\be}}$ because $n_{0}\le P<n_{1}$ for each possible $n_{0}$, 
\begin{eqnarray*}
\disp \sum_{j=0}^{+\8}\frac1{\left(1+\ga^{n_{0}-n_{1}-1}+\frac{j}{\ga^2}\right)^{\be}}&=&\frac1{\left(1+\ga^{n_{0}-n_{1}-2}\right)^{\be}}+\disp \sum_{j=1}^{+\8}\frac1{\left(1+\ga^{n_{0}-n_{1}-2}+\frac{j}{\ga^2}\right)^{\be}}\\
&\le & 1+\frac{\ga^2}{\be-1}\frac1{\left(1+\ga^{n_{0}-n_{1}-2}\right)^{\be-1}}\\
&\le & 1+\frac{\be-1}2,
\end{eqnarray*}
for $\be \ge \sqrt{2} \ga$.
 Finally, inequality \eqref{eq:est0} yields 
$$C_{E}\leq 2P\, \frac{N+1}{\be-1}\sum_{M=0}^{+\8}A^{M}\sum_{n=1}^{+\8}\ga^{n(1-\be)}\sum_{j=0}^{M-1}\frac{n^{j}}{j!},$$
with $\disp A:=\frac{\ga}{\be-1}K=\frac{\ga}{\be-1-\frac1{\log\ga}}$. This tends to $0$ as $\be \to +\8$. 
\end{proof}

\subsection{End of the proof of Theorem~\ref{theo-pt}}

\subsubsection{End of the case $-\log\frac{n+1}n$}
Proposition~\ref{prop-spectral&zc} shows that there exists some minimal $\be_{0}$ such that $\l_{0,\be}<1$ for every $\be>\be_{0}$. 
This also shows that $\CP(\be)=0$ for $\be>\be_{0}$. 
Since $\CP(\be)$ is a continuous and convex function, it is constant for 
$\be>\be_{0}$. As $\CP(0)=\log2$, there exists 
a minimal $\be_{c} > 0$ such that $\CP(\be)>0$ for every $0\le \be<\be_{c}$. 
Clearly, $\be_{c}\le \be_{0}$. 

We claim that for $\be<\be_{c}$, there exists a unique equilibrium state and that it has full support. Indeed, there exists at least one equilibrium state, say $\mu_{\be}$, and at least one cylinder, say $J$, has positive $\mu_{\be}$-measure. 
Therefore, we can induce on this cylinder, and the form of potential (see \cite[Theorem 4]{leplaideur-synth}) shows that there exists a unique local equilibrium state. It is a local Gibbs measure and therefore $\mu_{\be}$ is uniquely determined on each cylinder, and unique  and with full support (due to the mixing property). 

We claim that the pressure function $\CP(\be)$ is analytic on $[0,\be_{c}]$. 
Indeed, each cylinder $J$ has positive $\mu_{\be}$-measure and the associated $Z_{c}(\be)$ is the pressure of the dotted system (that is: restricted to the trajectories that avoid $J$). This set of trajectories has a pressure strictly smaller than $\CP(\be)$ because otherwise, several equilibrium states would coexist. 
Therefore $\CP(\be)$ is determined by the implicit equation $\l_{\CP(\be),\be}=1$ and $\CP(\be)>Z_{c}(\be)$ for $\be \in [0,\be_{c}]$. 
The Implicit Function Theorem shows that $\CP(\be)$ is analytic. 
 
For $\be\ge \be_{c}$, the pressure $\CP(\be)=0$ and for cylinders $J$ as above, we have $Z_{c}(\be)\ge 0$. This shows that $Z_{c}(\be)=0$ for every $\be\ge \be_{c}$. Due to the form of the potential, $\l_{0,\be}$ is continuous and decreasing  in $\be$. 

We claim that $\be_{c}=\be_{0}$. 
Indeed, assume by contradiction $\be<\be_{c}$. Then $\l_{0,\be_{c}}>1$, since  otherwise (because $\l_{0,\be}$ being strictly decreasing in $\be$), 
$\l_{0,\be_{c}}\le 1$ would yield that $\l_{0,\be}<1$ for every $\be>\be_{c}$.
This would imply $\be_{c} \ge \be_{0}$ (recall that $\be_{0}$ is minimal with this property). 
Now, for fixed $\be$, $Z\mapsto \l_{Z,\be}$ is continuous and strictly decreasing and goes to $0$ at $Z\to+\8$. Therefore, if $\l_{0,\be_{c}}>1$ then there exists $Z>0$ such that $\l_{Z,\be_{c}}=1$. The local equilibrium state for this $Z$ generates
 some $\s$-invariant probability measure\footnote{Since $Z_{c}(\be_{c})=\CP(\be_{c})=0<Z$, the expectation of the return time is comparable to $\left|\frac{\partial\CL_{Z,\be_{c}}(\BBone_{J})}{\partial Z}\right|$, which converges.} with pressure for $\be \varphi$ equal to $Z$, thus positive, and this is in contradiction with $\CP(\be_{c})=0$. 
This proves that $\be_{c}=\be_{0}$ and finishes the proof of Theorem~\ref{theo-pt} in the case that $V(x) = -\log \frac{n+1}{n}$. 

\subsection{The general case \boldmath $V\in \CX_{1}$ \unboldmath}
For $V \in \CX_1$, there exists $\kappa>0$ such that 
$$-V \le \kappa\varphi.$$
 This shows that the pressure function is constant equal to zero for $\be\ge \be_{0}/\kappa$. Again, the pressure is convex, thus non-increasing and continuous. We can define $\be'_{c}$ such that $\CP(\be)>0$ for$0\le\be\le \be'_{c}$ and $\CP(\be)=0$ for $\be\ge \be'_{c}$. 
 
 The rest of the argument is relatively similar to the previous discussion. We deduce that for $\be<\be'_{c}$, there exists a unique equilibrium state, it has full support and $\CP(\be)$ is analytic on this interval. For $\be\ge \be'_{c}$, it is not clear that $\l_{0,\be}$ decreases in $\be$. However,
we do not really need this argument, because if $\l_{0,\be}>1$, then the decrease of $Z\mapsto \l_{Z,\be}$ (which follows from convexity argument and $\disp\lim_{Z\to+\8}\l_{Z,\be}=0)$, is sufficient to produce a contradiction.

\bibliographystyle{plain}
\bibliography{mabibliohenk2}

Faculty of Mathematics\\
University of Vienna\\
Oskar Morgensternplatz 1, A-1090 Vienna\\ 
Austria\\
\texttt{henk.bruin@univie.ac.at}\\
\texttt{http://www.mat.univie.ac.at/$\sim$bruin}
\\[3mm]
D\'epartement de  Math\'ematiques\\
Universit\'e de Brest\\
6, avenue Victor Le Gorgeu\\
C.S. 93837, France \\
\texttt{Renaud.Leplaideur@univ-brest.fr}\\
\texttt{http://www.math.univ-brest.fr/perso/renaud.leplaideur}

\end{document}